\documentclass{amsart}

\usepackage[usenames]{color} 
\usepackage{amssymb} 
\usepackage{amsmath} 
\usepackage[nobysame]{amsrefs} 

\usepackage{enumerate}
\usepackage[colorlinks]{hyperref}
\usepackage{colonequals}
\usepackage{tikz-cd}

\makeatletter
\@namedef{subjclassname@2020}{\textup{2020} Mathematics Subject Classification}
\makeatother

\newtheorem{theorem}[subsection]{Theorem}
\newtheorem{proposition}[subsection]{Proposition}
\newtheorem{lemma}[subsection]{Lemma}
\newtheorem{corollary}[subsection]{Corollary}

\theoremstyle{definition}

\newtheorem{example}[subsection]{Example}
\newtheorem{question}[subsection]{Question}

\theoremstyle{remark}
\newtheorem{remark}[subsection]{Remark}

\numberwithin{equation}{subsection}

\newcommand{\ann}{\operatorname{ann}}
\newcommand{\Ass}{\operatorname{Ass}}

\newcommand{\bbZ}{\mathbb{Z}}

\newcommand{\Coker}{\operatorname{Coker}}

\newcommand{\ddual}[1]{{#1}^{\dagger}}
\newcommand{\dcat}[1]{\operatorname{D}(#1)}
\newcommand{\dco}[2]{\operatorname{L}\!\varLambda^{#1}(#2)}

\newcommand{\depth}{\operatorname{depth}}
\renewcommand{\dim}{\operatorname{dim}}

\newcommand{\End}{\operatorname{End}}
\newcommand{\Ext}[4]{\operatorname{Ext}^{#1}_{#2}(#3,#4)}
\newcommand{\fm}{\mathfrak{m}}

\newcommand{\fp}{\mathfrak{p}}
\newcommand{\ges}{\geqslant}
\newcommand{\hh}[1]{\operatorname{H}(#1)}
\newcommand{\HH}[2]{\operatorname{H}_{#1}(#2)}
\newcommand{\CH}[2]{\operatorname{H}^{#1}(#2)}
\newcommand{\height}{\operatorname{height}}
\newcommand{\Hom}[3]{\operatorname{Hom}_{#1}(#2,#3)}

\newcommand{\kos}[2]{K(#1;#2)}
\newcommand{\lam}[2]{\varLambda^{#1}{#2}}
\newcommand{\length}{\operatorname{length}}

\newcommand{\lra}{\longrightarrow}
\newcommand{\lotimes}[1]{\otimes_{#1}^{\operatorname{L}}}
\newcommand{\lch}[3]{\operatorname{H}^{#1}_{#2}(#3)}
\newcommand{\mcO}{\mathcal{O}}

\newcommand{\pdim}{\operatorname{proj\,dim}}

\newcommand{\rder}[1]{\operatorname{R}\!{#1}_*}
\newcommand{\Rhom}[3]{\operatorname{RHom}_{#1}(#2,#3)}
\newcommand{\rgam}[2]{\operatorname{R}\!\varGamma_{\!#1}(#2)}
\newcommand{\shift}{\mathsf{\Sigma}}
\newcommand{\spec}{\operatorname{Spec}}

\newcommand{\syzk}{\Omega}
\newcommand{\tors}[2]{\varGamma_{#1}{#2}}
\newcommand{\ul}{\underline}
\newcommand{\xra}{\xrightarrow}

\title[MCM complexes]{Maximal Cohen-Macaulay complexes and \\ their uses: A partial survey}

\author[Iyengar]{Srikanth B. Iyengar}
\address{Department of Mathematics, University of Utah, Salt Lake City, UT 84112, USA}
\email{iyengar@math.utah.edu}

\author[Ma]{Linquan Ma}
\address{Department of Mathematics, Purdue University, 150 N. University street, IN 47907}
\email{ma326@purdue.edu}

\author[Schwede]{Karl Schwede}
\address{Department of Mathematics, University of Utah, Salt Lake City, UT 84112, USA}
\email{schwede@math.utah.edu}

\author[Walker]{Mark E. Walker}
\address{Department of Mathematics, University of Nebraska, Lincoln, NE 68588, U.S.A.}
\email{mark.walker@unl.edu}

\begin{document}

\begin{abstract}
This work introduces a notion of complexes of maximal depth, and maximal Cohen-Macaulay complexes,  over a commutative noetherian local ring. The existence of such complexes is closely tied to the Hochster's ``homological conjectures", most of which were recently settled by Andr\'e. Various constructions of maximal Cohen-Macaulay complexes are described, and their existence is applied to give new proofs of some of the homological conjectures, and also of certain results in birational geometry.
\end{abstract}

\date{15th June 2021}



\thanks{Part of this work was partly supported by  National Science Foundation DMS grants 2001368 (SBI);  1901672 and FRG grant 1952366 (LM),  1801849 and FRG grant 1952522 (KS), and 1901848 (MW). LM was supported by a fellowship from the Sloan Foundation.  KS was supported by a fellowship from the Simons foundation.}

\maketitle

\section*{Introduction}
Big Cohen-Macaulay modules over (commutative, noetherian) local rings  were introduced by Hochster around fifty years ago and  their relevance to local algebra is established beyond doubt. Indeed, they play a prominent role in Hochster's lecture notes~\cite{Hochster:1975}, where he describes a number of homological conjectures that can be proved using big Cohen-Macaulay modules, and their finitely generated counterparts, the maximal Cohen-Macaulay modules; the latter are sometimes called, as in \emph{loc.~cit.}, ``small" Cohen-Macaulay modules.  Hochster~\cite{Hochster:1973, Hochster:1975} proved that big Cohen-Macaulay modules exist when the local ring contains a field; and conjectured that even rings of mixed-characteristic possess such modules.  This conjecture was proved by Andr\'e~\cite{Andre:2018}, thereby settling a number of the homological conjectures.  In fact, by results of Hochster and Huneke \cite{Hochster/Huneke:1992, Hochster/Huneke:1995}, and Andr\'e~\cite{Andre:2018} there exist even big Cohen-Macaulay \emph{algebras} over any local ring.  The reader will find a survey of these developments in \cite{Hochster:2017, Ma/Schwede:2019}.

In this work we introduce three versions of the Cohen-Macaulay property that apply also to complexes of modules, discuss various constructions that give rise to them, and  present some consequences that follow from their existence.  In fact such complexes have come up earlier, in the work of Roberts~\cite{Roberts:1980, Roberts:1980a},  recalled in \S\ref{ss:resolve-singularities}, and in recent work of Bhatt~\cite{Bhatt:2020}, though only in passing.  What we found is that results that were proved using big Cohen-Macaulay modules can often be proved using one of their complex variants. This assertion is backed up the material presented in Sections~\ref{se:max-depth} and \ref{se:birational}. Moreover, as will be apparent in the discussion in Section~\ref{se:mcm}, the complex versions are easier to construct, and with better finiteness properties. It thus seems worthwhile to shine an independent light on them. Let us begin by defining them.

We say that a complex $M$ over a local ring $R$ with maximal ideal $\fm$ has \emph{maximal depth} if $\depth_RM= \dim R$, where depth is as in \S\ref{ss:depth}; we ask also that $\hh M$ be bounded and  the canonical map $\HH 0M\to \HH 0{k\lotimes RM}$ be non-zero. Any complex that satisfies the last condition has depth at most $\dim R$, whence the name ``maximal depth".  An $R$-module has maximal depth precisely when it is big Cohen-Macaulay. The depth of a complex can be computed in terms of its local cohomology modules, $\lch i{\fm}M$, with support on $\fm$. Thus $\depth_RM=\dim R$ means that $\lch i{\fm}M$ is zero for $i<\dim R$, and nonzero for $i=\dim R$.  A complex of maximal depth is \emph{big Cohen-Macaulay} if $\lch i{\fm}M=0$ for $i>\dim R$ as well. When in addition the $R$-module $\hh M$ is finitely generated, $M$ is \emph{maximal
  Cohen-Macaulay} (MCM). Thus an MCM module is what we know it to be. These notions are discussed in detail in Section~\ref{se:max-depth} and \ref{se:mcm}.

When $R$ is an excellent local domain with residue field of positive characteristic, $R^{+}$, its integral closure  in an algebraic closure of its field of fractions, is big Cohen-Macaulay. This was proved by Hochster and Huneke~\cite{Hochster/Huneke:1992}, see also  Huneke and Lyubeznik~\cite{Huneke/Lyubeznik:2007}, when $R$ itself contains a field of positive characteristic.  When $R$ has mixed characteristic this is a recent result of Bhatt~\cite{Bhatt:2020}. Thus for such rings there is a canonical construction of a big Cohen-Macaulay module, even an \emph{algebra}.  See also the work of Andr\'e~\cite{Andre:2020} and Gabber~\cite{Gabber:2018} concerning functorial construction of big Cohen-Macaulay algebras; see also \cite[Appendix A]{Ma/Schwede/Tucker/Waldron/Witaszek:2019}). On the other hand, $R^{+}$ is never big Cohen-Macaulay when $R$ contains the rationals and is a normal domain of Krull dimension at least $3$, by a stadard trace argument. As far as we know,  in this context there are no such ``simple" models of big Cohen-Macaulay modules, let alone algebras. See however Schoutens' work~\cite{Schoutens:2004}.

When $R$ is essentially of finite type over a field of characteristic zero, the derived push-forward of the structure sheaf of a resolution of singularities of $\spec R$ is an MCM complex~\cite{Roberts:1980}. What is more, this complex is equivalent to a graded-commutative differential graded algebra, as we explain in \ref{ss:resolve-singularities}. This is noteworthy because when such a ring $R$ is also a normal domain of dimension $\ge 3$ it cannot have any MCM algebras, by the same trace argument as for $R^{+}$. For any local ring $R$ with a dualizing complex, there are simple constructions of MCM complexes---see Corollaries \ref{co:MCM=CEC} and \ref{co:omega}---though we do not know any that are also differential graded algebras.  In \cite{Bhatt:2014} Bhatt gives examples of complete local rings, containing a field of positive characteristic, that do not have any MCM algebras.

As to applications, in Section~\ref{se:max-depth}  we prove the New Intersection Theorem and its improved version using complexes of maximal depth, extending the ideas from \cite{Iyengar:1999} where they are proved using big Cohen-Macaulay modules.  Hochster proved in \cite{Hochster:1983} that the Improved New Intersection Theorem is equivalent to the Canonical Element Theorem.  In Section~\ref{se:mcm} we use results from \emph{loc.~cit.} to prove that for local rings with dualizing complexes the Canonical Element Theorem implies the existence of MCM complexes. An interesting point emerges: replacing ``module" with ``complex" puts the existence of big Cohen-Macaulay modules on par with the rest of the homological conjectures.

In Section~\ref{se:birational} we paraphrase Boutot's proof of his theorem on rational singularities to highlight the role of MCM complexes. We also give a new proof of a subadditivity property for  multiplier ideals. On the other hand, there are applications of MCM modules that do require working with modules; see \ref{re:mcm-module}. Nevertheless, it is clear to us that big Cohen-Macaulay complexes and MCM complexes have their uses, hence this survey.

\section{Local cohomology and derived completions}
\label{se:lch}
In this section, we recall basic definitions and results on local cohomology and derived completions. Throughout $R$ will be a commutative noetherian ring. By an \emph{$R$-complex} we mean a complex of $R$-modules; the grading will be upper or lower, depending on the context. In case of ambiguity, we indicate the grading; for example, given an $R$-complex $M$, the supremum of $\hh M$ depends on whether the grading is upper or lower. So we write $\sup\HH{*}M$ for the largest integer $i$ such that $\HH iM\ne 0$, and $\sup\CH{*}M$ for the corresponding integer for the upper grading.

We write $\dcat R$ for the (full) derived category of $R$ viewed as a triangulated category with translation $\shift$, the usual suspension functor on complexes. We take \cite{Lipman:1999, Dwyer/Greenlees:2002} as basic references, augmented by \cite{Avramov/Foxby:1991, Roberts:1980}, except that we use the term ``semi-injective" in place of ``q-injective" as in \cite{Lipman:1999}, and ``DG-injective", as in \cite{Avramov/Foxby:1991}. Similarly for the projective and flat analogs.

\subsection{Derived $I$-torsion}
Let $I$ an ideal in $R$.  The \emph{$I$-power torsion} subcomplex of an $R$-complex $M$ is
\[
\tors IM\colonequals \{m\in M\mid I^nm=0\text{ for some $n\ge 0$}\}.
\]
By $m\in M$ we mean that $m$ is in $M_i$ for some $i$. The corresponding derived functor is denoted $\rgam IM$; thus $\rgam IM = \tors I{J}$ where $M\xra{\sim} J$ is any semi-injective resolution of $M$. In fact, one can compute these derived functors from any complex of injective $R$-modules quasi-isomorphic to $M$; see \cite[\S3.5]{Lipman:1999}. By construction there is a natural morphism $\rgam IM \to M$ in the $\dcat R$.
The $R$-modules
\[
\lch iIM\colonequals \CH i{\rgam IM} \quad\text{for $i\in\bbZ$}
\]
are the local cohomology modules of $M$, supported on $I$. Evidently, these modules are $I$-power torsion. Conversely, when the $R$-module $\hh M$ is $I$-power torsion, the natural map $\rgam IM\to M$ is an isomorphism in $\dcat R$; see \cite[Proposition~6.12]{Dwyer/Greenlees:2002}, or \cite[Corollary~3.2.1]{Lipman:1999}.

In what follows we will use the fact that the class of $I$-power torsion complexes form a localizing subcategory of $\dcat R$; see \cite[\S6]{Dwyer/Greenlees:2002}, or \cite[\S3.5]{Lipman:1999}. This has the consequence that these complexes are stable under various constructions. For example, this class of complexes is closed under $L\lotimes R(-)$ for any $L$ in $\dcat R$. Thus, for any $R$-complexes $L$ and $M$ the natural map
\begin{equation}
\label{eq:lch-tensor}
\rgam I{L\lotimes RM} \lra L\lotimes R \rgam IM
\end{equation}
is a quasi-isomorphism.

\subsection{Derived $I$-completion}
The $I$-adic completion of an $R$-complex $M$ with respect to the ideal $I$, denoted $\lam IM$, is
\[
\lam IM \colonequals \lim_{n\ges 0} M/I^nM\,.
\]
This complex is thus the limit of the system
\[
\cdots \lra M/I^{n+1}M\lra M/I^nM\lra \cdots \lra M/IM\,.
\]
The canonical surjections $M\to M/I^nM$ induce an $R$-linear map $M\to \lam IM$. If this is an isomorphism we say that $M$ is \emph{$I$-adically complete}, or just \emph{$I$-complete}, though we reserve this name mainly for modules. The \emph{left-derived completion} with respect to $I$ of an $R$-complex $M$ is the $R$-complex
\[
\dco IM\colonequals \lam IP \quad \text{where $P\simeq M$ is a semi-projective resolution.}
\]
This complex is well-defined in $\dcat R$, and there is a natural morphism
\[
M\lra \dco IM\,.
\]
We say $M$ is \emph{derived $I$-complete} if this map is a quasi-isomorphism; equivalently if each $\HH iX$ is derived $I$-complete; see \cite[Proposition~6.15]{Dwyer/Greenlees:2002}, or \cite[Tag091N]{Stacks}

The derived $I$-complete modules from a colocalizing subcategory of $\dcat R$, and this means that for $N$ in $\dcat R$ the  natural map
\[
\dco I{\Rhom RNM}  \lra \Rhom RN{\dco IM}
\]
is a quasi-isomorphism. In particular, when $F$ is a perfect complex, we have an isomorphim in $\dcat R$
\begin{equation}
\label{eq:dco-tensor}
F\lotimes R\dco IM \simeq \dco I{F\lotimes RM}.
\end{equation}
These isomorphisms will be useful in what follows. It is a fundamental fact, proved by Greenlees and May~\cite{Greenlees/May:1992a}, see also
\cite[Proposition~4.3]{Dwyer/Greenlees:2002} or \cite[\S4]{Lipman:1999}, that derived local cohomology and derived completions  are adjoint functors:
\begin{equation}
\label{eq:GM}
\Rhom R{\rgam IM}N \simeq \Rhom R M{\dco IN}\,.
\end{equation}
One can take this as a starting point for defining derived completions, which works better in the non-noetherian settings; see \cite{Stacks}. This adjunction implies that the natural maps are quasi-isomorphisms:
\begin{equation}
\label{eq:GM2}
\dco I{\rgam IM} \xra{\ \simeq \ } \dco IM \quad\text{and}\quad  \rgam IM \xra{\ \simeq \ }  \rgam I{\dco IM}\,.
\end{equation}

The result below, due to A.-M.~Simon~\cite[1.4]{Simon:1990},  is a version of Nakayama's Lemma for cohomology of complete modules.
It is clear from the proof that we only need $X$ to be derived $I$-complete; see \cite[Tag09b9]{Stacks}.

\begin{lemma}
\label{le:complete-nak}
For any $R$-complex $X$ consisting of $I$-complete modules, and integer $i$, if $I\HH iX= \HH iX$, then $\HH iX=0$.
\end{lemma}

\begin{proof}
The point is that $Z_i$, the module of cycle in degree $i$, is a closed submodule of the $I$-complete module $X_i$, and hence is also $I$-complete.  Moreover $\HH iX$ is the cokernel of the map $X_{i+1} \to Z_i$, and a map between $I$-complete modules is zero if and only if its $I$-adic completion is zero. This translates to the desired result.
\end{proof}

\subsection{Koszul complexes}
Given a sequence of elements $\ul r\colonequals r_1,\dots,r_n$ in the ring $R$, and an $R$-complex $M$, we write $\kos {\ul r}M$ for the Koszul complex on $\ul r$ with coefficients in $M$, namely
\[
\kos{\ul r}M \colonequals \kos{\ul r}R \otimes_RM\,.
\]
Its homology is denoted $\HH{*}{\ul r;M}$. For a single element $r\in R$, the complex $\kos rM$ can be constructed as the mapping cone of the homothety map $M\xra{r}M$. In particular, one has an exact sequence

\begin{equation}
\label{eq:mapping-cone}
0\lra M \lra \kos rM \lra \shift M \lra 0
\end{equation}
of $R$-complexes. The Koszul complex on a sequence can thus be constructed as an iterated mapping cone. From Lemma~\ref{le:complete-nak} one gets the result below. Recall that $\sup\HH{*}{-}$ denotes the supremum, in lower grading.

\begin{lemma}
\label{le:sup-Koszul}
Let $R$ be a noetherian ring and $X$ a derived $I$-complete $R$-complex. For any sequence $\ul r\colonequals r_1,\dots,r_n$ in $I$ one has
\[
\sup \HH{*}{\ul r; X}\ge \sup \HH{*}X\,.
\]
\end{lemma}

\begin{proof}
When $X$ is derived $I$-complete so is $\kos rX$ for any $r\in I$. It thus suffices to verify the desired claim for $n=1$. Replacing $X$ by $\lam IP$, where $P$ is a semi-projective resolution of $X$, we can assume $X$, and hence also $\kos rX$, consists of $I$-complete modules.  The desired inequality is then immediate from the standard long exact sequence in homology
\[
\cdots \lra \HH iX \xra{\ r\ } \HH iX \lra \HH i{r,X} \lra \HH {i-1}X\lra\cdots
\]
arising from the mapping cone sequence \ref{eq:mapping-cone} and Lemma~\ref{le:complete-nak}.
\end{proof}

To wrap up this section we recall the notion of depth for complexes.

\subsection{Depth}
\label{ss:depth}
The \emph{$I$-depth} of an $R$-complex $M$  is
\[
\depth_R(I,M)\colonequals \inf\{i\mid \lch iIM\ne 0\}.
\]
In particular, $\depth_R(I,M)=\infty$ if $\lch {}IM=0$. When the ring $R$ is local, with maximal ideal $\fm$, the \emph{depth} of $M$ refers to the $\fm$-depth of $M$.

Depth can also be computed using Ext and Koszul homology:
\[
\depth_R(I,M) = \inf\{i\mid \Ext iR{R/I}M\ne 0\}\,,
\]
and if a sequence $\ul r\colonequals r_1,\dots,r_n$ generates $I$, then
\[
\depth_R(I,M) = n -  \sup\{i\mid \HH i{\ul r; M}\ne 0\}
\]
This last equality can be expressed in terms of Koszul cohomology. All these results are from \cite{Foxby/Iyengar:2001}, though special cases (for example, when $M$ is an $R$-module) had been known for much longer.

\begin{remark}
\label{re:depth-properties}
Let $R$ be a commutative ring, $I$ an ideal in $R$, and $M$ an $R$-complex. Set $s=\sup\HH {*}M$.
\begin{enumerate}
\item
$\depth_R(I,M) \ge -s$ and equality holds if  $\tors I{(\HH sM)}\ne 0$.
\item
When $R$ is local and  $F$ is a finite free complex, one has
\[
\depth_R(F\otimes_RM) = \depth_R M - \pdim_RF
\]
\end{enumerate}
For part (1) see \cite[2.7]{Foxby/Iyengar:2001}. When $F$ is the resolution of a module and $M=R$, part (2) is nothing but the equality of Auslander and Buchsbaum.  For a proof in the general case see, for example, \cite[Theorem~2.4]{Foxby/Iyengar:2001}.
\end{remark}

\section{Complexes of maximal depth and the intersection theorems}
\label{se:max-depth}
In this section we introduce a notion of ``maximal depth" for
complexes  over local rings.  The gist of the results presented here
is that their existence implies the Improved New Intersection Theorem,
and hence a whole slew of ``homological conjectures", most of which
have been  recently settled by Andr\'e~\cite{Andre:2018}.

A \emph{module} of maximal depth is nothing but a big Cohen-Macaulay module and Hochster  proved, already in \cite{Hochster:1975}, that their existence implies the homological
conjectures mentioned above. On the other hand, the Canonical Element Conjecture, now theorem, implies that $R$ has a \emph{complex} of maximal depth, even one with finitely
generated homology. This will be one of the outcomes of the discussion in the next section; see Remark~\ref{re:MCM=CEC2}. No such conclusion can be drawn about big Cohen-Macaulay
modules.

\subsection{Complexes of maximal depth}
\label{ss:max-depth}
Throughout $(R,\fm,k)$ will be a local ring, with maximal ideal $\fm$ and residue field $k$.  We say that an $R$-complex $M$ has \emph{maximal depth} if the following conditions
hold:
\begin{enumerate}[\quad\rm(1)]
\item
$\hh M$ is bounded;
\item
$\HH 0M\to \HH 0{k\lotimes RM}$ is nonzero; and
\item
$\depth_R M= \dim R$.
\end{enumerate}
The nomenclature is based on that fact that $\depth_RM\le \dim R$ for any complex $M$ that satisfies condition (2) above.  This inequality follows from Lemma~\ref{le:cm-hh0} applied with $F\colonequals K$, the Koszul complex on a system of parameters for $R$. Condition (3) can be restated as
\begin{equation}
\label{eq:lch-cm}
\lch i{\fm}M = 0 \quad\text{for $i<\dim R$} \quad\text{and} \quad \lch {\dim R}{\fm}M \ne 0\,.
\end{equation}
Clearly when $M$ is a module it has maximal depth precisely when it is big Cohen-Macaulay; condition (2) says that $M\ne \fm M$. Note also that if a complex $M$ has maximal depth then so does $M\oplus \shift^{-n}N$ for any $R$-module $N$ and integer $n\ge \dim R$.

\begin{lemma}
\label{le:cm-hh0}
Let $M$ be an $R$-complex with the natural map $\HH 0M\to \HH 0{k\lotimes RM}$ nonzero. For any $R$-complex $F$ with $\HH iF=0$ for $i<0$, if $\HH 0F\otimes_R k$ is nonzero, then so is $\HH 0{F\lotimes RM}$.
\end{lemma}

\begin{proof}
We can assume $M$ is semi-projective, so the functor $-\lotimes RM$ is represented by $-\otimes_RM$. By hypothesis there exists a cycle, say $z$, in $M_0$ whose image in $k\otimes_RM=M/\fm M$ is not a boundary. Consider the morphism $R\to M$ of $R$-complexes, where $r\mapsto rz$.
Its composition $R\to M\to k\otimes_RM$ factors through the canonical surjection $R\to k$, yielding the commutative square
\[
\begin{tikzcd}
R \arrow[r] \arrow[d] & M \arrow[d] \\
k  \arrow[r,shift left=.5ex] & k\otimes_{R}M. \arrow[l, dashrightarrow,shift left=.5ex]
\end{tikzcd}
\]
The dotted arrow is a left-inverse in $\dcat R$ of the induced $k\to k\otimes_RM$. It exists because $k\to \hh{k\otimes_RM}$ is nonzero, by the choice of $z$, and the complex $k\otimes_RM$ is quasi-isomorphic to $\hh{k\otimes_RM}$ in $\dcat k$, and hence in $\dcat R$.
Applying $F\lotimes R-$ to the diagram above yields the commutative square in $\dcat R$ on the left:
\[
\begin{gathered}
\begin{tikzcd}
F \arrow[r] \arrow[d] & F\lotimes RM \arrow[d] \\
F\lotimes Rk  \arrow[r,shift left=.5ex] & F\lotimes R(k\otimes_RM) \arrow[l, dashrightarrow,shift left=.5ex]
\end{tikzcd}
\end{gathered}
\qquad
\begin{gathered}
\begin{tikzcd}
\HH 0F \arrow[r] \arrow[d]& \HH 0{F\lotimes RM} \arrow[d] \\
\HH 0{F\lotimes Rk}  \arrow[r,shift left=.5ex] & \HH 0{F\lotimes R(k\otimes_RM)} \arrow[l, dashrightarrow,shift left=.5ex]
\end{tikzcd}
\end{gathered}
\]
The commutative square on the right is obtained by applying $\HH 0{-}$ to the one on the left. In this square, the hypotheses on $F$ imply that the vertical map on the left is nonzero, so hence is its composition with the horizontal arrow. The commutativity of the square then yields that $\HH 0{F\lotimes RM}$ is nonzero.
\end{proof}

The following result is due to Hochster and Huneke for rings containing a field, and due to Andr\'e in the mixed characteristic case.

\begin{theorem}[Andr\'e~\cite{Andre:2018}, Hochster\,\&\, Huneke~\cite{Hochster/Huneke:1992, Hochster/Huneke:1995}]
\label{th:big}
Each noetherian local ring possesses a big Cohen-Macaulay algebra. \qed
\end{theorem}

As has been said before, the existence of big Cohen-Macaulay algebras, and hence big Cohen-Macaulay modules, implies many of the homological conjectures. In particular, it can be used to give a quick proof of the New Intersection Theorem, first proved in full generality by P.~Roberts~\cite{Roberts:1987} using intersection theory; see also \cite{Piepmeyer/Walker:2009}. Here is a proof that uses only the existence of a \emph{complexes} of maximal depth; the point being that they are easier to construct than big Cohen-Macaulay modules. Our argument is modeled on that of \cite[Theorem~2.5]{Iyengar:1999}, which uses big Cohen-Macaulay modules.

\begin{theorem}
Let $R$ be a local ring. Any finite free $R$-complex
\[
F\colonequals 0\to F_n\to\cdots\to F_0\to 0
\]
with $\HH 0F\ne 0$ and $\length_R\HH iF$ finite for each $i$ satisfies $n\ge \dim R$.
\end{theorem}

\begin{proof}
Let $M$ be an $R$-complex of maximal depth. As $\hh F$ is of finite length, the $R$-module $\hh{F\otimes_RM}$ is $\fm$-power torsion, so \ref{re:depth-properties}(1) yields the second equality:
\begin{align*}
\pdim_RF
	&=  \depth_R M - \depth_R (F\otimes_RM) \\
	& =\depth_RM  + \sup\HH{*}{F\otimes_RM}\\
	& \ge \depth_RM\\
	& = \dim R
\end{align*}
The first one is by~\ref{re:depth-properties}(2). The  inequality is by Lemma~\ref{le:cm-hh0}, noting that $\HH 0F\otimes_Rk$ is nonzero by Nakayama's lemma.
\end{proof}

One can deduce also the Improved New Intersection Theorem~\ref{th:INIT} from the existence of complexes of maximal depth, but the proof takes some more preparation.

\begin{lemma}
\label{le:cm-dc}
Let $R$ be a local ring and $M$ an $R$-complex. If $M$ has maximal depth, then so does $\dco IM$ for any ideal $I\subset R$.
\end{lemma}

\begin{proof}
Condition (1) for maximal depth holds because $\hh M$ bounded implies $\hh{\dco IM}$ is bounded; this follows, for example, from \eqref{eq:GM} and the observation  $\rgam IR$ has finite projective dimension. As to the other conditions,  the main point is that for any $R$-complex $X$ such that $\hh X$ is $I$-power torsion,  the  canonical map $M\to \dco IM$ in $\dcat R$ induces a quasi-isomorphism
\[
X\lotimes RM \xra{\ \simeq \ } X\lotimes R\dco IM.
\]
This can be deduced from \eqref{eq:dco-tensor} and \eqref{eq:GM2}. In particular, taking $X=\rgam {\fm}R$, where $\fm$ is the maximal ideal of $R$, yields
\[
\rgam \fm M\simeq \rgam {\fm} {\dco IM}\,,
\]
so that $\depth_RM=\depth_R\dco IM$. Moreover, taking $X=k$  gives the isomorphism in the following commutative diagram in $\dcat R$:
\[
\begin{tikzcd}
M \arrow[d] \arrow[r] & k\lotimes R M \arrow[d,"\simeq"] \\
\dco IM \arrow[r]  & k\lotimes R\dco IM
\end{tikzcd}
\]
that is induced by the morphism $M\to \dco IM$. Since $M$ has maximal depth, the map in the top row is nonzero when we apply $\HH 0{-}$, and so the same holds for the map in the bottom row. Thus $\dco IM$ has maximal depth.
\end{proof}

\begin{lemma}
\label{le:cm-localize}
Let $(R,\fm,k)$ is a local ring and $M$ a  derived $\fm$-complete $R$-complex of maximal depth. Set $d\colonequals \dim R$. The following statements hold:
\begin{enumerate}[\quad\rm(1)]
\item
For any system of parameters $r_1,\dots,r_d$ for $R$, one has
\[
\depth_R( \kos{r_1,\dots,r_n}M ) =n\quad\text{for each $1\le n\le d$.}
\]
In words, the depth of $M$ with respect to the ideal $(r_1,\dots,r_n)$ is $n$.
\item
For any $\fp\in\spec R$ one has
\[
\depth_{R_{\fp}}M_{\fp}\ge \dim R_{\fp}\,,
\]
and equality holds when the map $\HH 0M\to \HH 0{k(\fp)\lotimes RM}$ is nonzero, in which case the $R_\fp$-complex $M_\fp$ has maximal depth.
\end{enumerate}
\end{lemma}

\begin{proof}
(1) Set $\ul r=r_1,\dots,r_d$. The hypothesis that $M$ has maximal depth and the depth sensitivity of the Koszul complex $\kos{\ul r}R$ yield $\HH i{\ul r;M}=0$ for $i\ge 1$. One has an isomorphism of $R$-complexes
\[
\kos{\ul r}M\cong \kos{r_{n+1},\dots,r_d}{\kos{r_1,\dots,r_n}M}\,.
\]
Since $M$ is derived complete with respect to $\fm$, it follows from Lemma~\ref{le:sup-Koszul}, applied to the sequence  $r_{n+1},\dots,r_d$ and $X\colonequals \kos{r_1,\dots,r_n}M$, that
\[
\HH i{\kos{r_1,\dots,r_n}M} = 0 \quad\text{for $i\ge 1$.}
\]
On the other hand, since the natural map $\HH 0M\to \HH 0{k\lotimes RM}$ is nonzero, Lemma~\ref{le:cm-hh0} applied with $F=\kos{r_1,\dots,r_n}R$, yields
\[
\HH 0{\kos{r_1,\dots,r_n}M}\ne 0\,.
\]
Thus the depth sensitivity of $\kos{r_1,\dots,r_n}M$ yields the equality in (1).

(2) Set $h\colonequals \height\fp$ and choose a system of parameters $\ul r\colonequals r_1,\dots,r_d$ for $R$ such that the elements $r_1,\dots,r_h$ are in $\fp$.   One has
\[
\depth_{R_{\fp}}\! M_\fp \ge  \depth_R((r_1,\dots,r_h),M)\ge h\,.
\]
where the first inequality is clear and the second one holds by (1). The natural map $M\to k(\fp)\lotimes RM$ factors through $M_\fp$, so under the additional hypothesis Lemma~\ref{le:cm-hh0} implies $\depth_{R_\fp}M_\fp \le h$. We conclude that $M_\fp$ has maximal depth.
\end{proof}

Given the preceding result, we argue as in the proof of \cite[Theorem~3.1]{Iyengar:1999} to deduce the Improved New Intersection Theorem:

\begin{theorem}
\label{th:INIT}
Let $R$ be a noetherian local ring  and $F\colonequals 0\to F_n\to\cdots\to F_0\to 0$ a finite free $R$-complex with $\HH 0F\ne 0$ and $\length\HH iF$ finite for each $i\ge 1$. If an ideal $I$ annihilates a minimal generator of $\HH 0F$, then $n\ge \dim R - \dim (R/I)$.
\end{theorem}

\begin{proof}
Let $M$ be an $R$-complex of maximal depth. By  Lemma~\ref{le:cm-dc}, we can assume $M$ is derived $\fm$-complete, so Lemma~\ref{le:cm-localize} applies. Set $s\colonequals \sup\HH{*}{F\otimes_RM}$ and note that $s\ge 0$, by Lemma~\ref{le:cm-hh0}.

Fix $\fp$ in $\Ass_R\HH s{F\otimes_RM}$, so that $\depth_{R_\fp}\HH s{F\otimes_RM}_\fp=0$. The choice of $\fp$ implies that $\hh{F\otimes_RM}_\fp$ is nonzero, and hence $\hh{F}_\fp$ and $\hh{M}_\fp$ are nonzero as well. Therefore one gets
\begin{equation}
\label{eq:init-proof}
\begin{aligned}
\pdim_{R_\fp}F_\fp &= \depth_{R_{\fp}}\! M_\fp - \depth_{R_\fp}(F\otimes_RM)_\fp \\
    & = \depth_{R_{\fp}}\! M_\fp + s \\
    &\ge  \dim R_\fp  +  s
   \end{aligned}
\end{equation}
The equalities are by \ref{re:depth-properties} and the inequality is by Lemma~\ref{le:cm-localize}(2).

Suppose $s\ge 1$. We claim that $\fp=\fm$, the maximal ideal of $R$, so \eqref{eq:init-proof} yields
\[
\pdim_RF\ge \dim R\,,
\]
which implies the desired inequality.

Indeed if $\fp\ne \fm$, then since $\length_R\HH iF$ is finite for $i\ge 1$, one gets that  $F_\fp\simeq {\HH 0F}_\fp$, which justifies the equality below:
\[
\depth R_\fp \ge \pdim_{R_\fp}{\HH 0F}_\fp =\pdim_{R_\fp}F_\fp \ge \dim R_\fp +s
\]
The first inequality is a consequence of the Auslander-Buchsbaum equality~\ref{re:depth-properties}(2), the second one is from \eqref{eq:init-proof}.  We have arrived at a contradiction for $s\ge 1$.

It remains to consider the case $s=0$. Set $X\colonequals F\otimes_RM$. Since $\HH 0F$ is finitely generated, Nakayama's Lemma and
Lemma~\ref{le:cm-hh0} imply that each minimal generator of $\HH 0F$ gives a nonzero element in $\HH 0X$. One of these is thus annihilated by $I$, by the hypotheses. Said otherwise, $\tors I{\HH 0F}\ne 0$.  Since $\sup \HH *X=0$, this implies $\depth_R(I, X)=0$, by Remark~\ref{re:depth-properties}, and hence one gets the equality below
\[
\depth_R X \leq \depth_R(I,X) + \dim (R/I) = \dim (R/I)
\]
The inequality can be verified by arguing as in the proof of \cite[Proposition 5.5(4)]{Iyengar:1999}: Let  $\ul a\colonequals a_1,\dots,a_l$ be a set of generators for the ideal $I$, and let $\ul b\colonequals b_1,\dots,b_n$ be elements in $R$ whose residue classes in $R/I$ form a system of parameters.  Since $M$ is derived $\fm$-complete, so is  $X$ and hence also $\kos{\ul a}X$. Then Lemma~\ref{le:sup-Koszul} applied to the sequence $\ul b$ and complex $\kos{\ul a}X$ yields
\[
\sup \HH *{\ul a, \ul b; X}\ge \sup \HH *{\ul a; X}\,;
\]
this gives the desired inequality. Finally it remains to invoke the Auslander-Buchsbaum equality once again to get
\[
\pdim_R F = \depth_R M - \depth_RX \ge \dim R - \dim (R/I)\,.
\]
This completes the proof.
\end{proof}

\section{MCM complexes}
\label{se:mcm}
In this section we introduce two strengthenings of the notion of complexes of maximal depth, and discuss various constructions that yield such complexes.
As before let $(R,\mathfrak m,k)$ be a local ring, of Krull dimension $d$.

\subsection{Big Cohen-Macaulay complexes}
\label{ss:mcm-complex}
We say that an $R$-complex $M$ is \emph{big Cohen-Macaulay} if the following conditions hold:
\begin{enumerate}[\quad\rm(1)]
\item
$\hh M$ is bounded;
\item
$\CH 0M\to \CH 0{k\lotimes RM}$ is nonzero.
\item
$\lch i{\fm}M= 0$ for $i\ne \dim R$;
\end{enumerate}
If in addition $\hh M$ is finitely generated, $M$ is \emph{maximal Cohen-Macaulay}; usually abbreviated to MCM.  Condition (2) implies in particular that $\CH 0{k\lotimes RM}$ is nonzero, and from this it follows that $\lch i{\fm}M\ne 0$ for some $i$. Thus condition (3)  implies $\depth_RM= \dim R$; in particular, a big Cohen-Macaulay complex has maximal depth, in the sense of \ref{ss:max-depth} and $\lch{\dim R}{\fm}M\ne 0$. However (3) is more restrictive, as the following observation shows.

\begin{lemma}
If $M$ is an MCM $R$-complex, then $\CH iM=0$ for $i\not\in[0,\dim R]$; moreover, $\CH 0M\ne 0$.
\end{lemma}

\begin{proof}
The last part of the statement is immediate from condition (2).

Set $d=\dim R$. Let $K$ be the Koszul complex on a system of parameters for $R$. Then one has isomorphisms
\[
K\otimes_RM \simeq K\lotimes R \rgam {\fm}M \simeq K\lotimes R\shift^{-d} \lch d{\fm}M
\]
where the first one is from \eqref{eq:lch-tensor}, since $K\otimes_RM$ is $\fm$-power torsion, and the second isomorphism holds by the defining property (3) of a big Cohen-Macaulay complex. Hence
\[
\inf \CH*{K\otimes_RM}\ge 0\quad\text{and}\quad \sup\CH *{K\otimes_RM}\le d\,.
\]
By our hypotheses, the $R$-module $\CH iM$ is finitely generated for each $i$, and since $K$ is a Koszul complex on $d$ elements, a standard argument leads to the desired vanishing of $\CH iM$.
\end{proof}

Any nonzero MCM $R$-module is MCM when viewed as complex. However, even over Cohen-Macaulay rings, which are not fields, there are MCM complexes that are not modules; see the discussion in \eqref{ss:max-depth}. In the rest of this section we discuss various ways MCM complexes can arise, or can be expected to arise. It turns out that often condition (2) is the one that is hardest to verify. Here is one case when this poses no problem; see \ref{ss:resolve-singularities} for an application. The main case of interest is where $A$ is a dg (=differential graded) $R$-algebra.

\begin{lemma}
\label{le:mcm-dga}
Let $A$ be an $R$-complex with a unital (but not necessarily associative) multiplication rule such that the Leibniz rule holds and  $i\colonequals\inf\HH {*}A$ is finite. If $\HH iA$ is finitely generated, then the identity element of $A$ is nonzero in $\HH  0{A \lotimes R k}$.
\end{lemma}

\begin{proof}
One has $\HH i{A \lotimes Rk}\cong \HH iA\otimes_Rk$ and the latter module is nonzero, by Nakayama's lemma and the finite generation hypothesis. We have $A \lotimes R k = A \otimes_R T$ where $T$ is a Tate resolution of $k$; see \cite{Tate:1957}. So $A \otimes_R T$ is also a (possibly non-associative) dg algebra. Thus if the identity element were trivial in $\hh{A \otimes_R T}$, then $\hh{A \otimes_R T}=0$ holds, contradicting $\HH i{A\lotimes Rk}\ne 0$.
\end{proof}

The MCM property for complexes has a simple interpretation in terms of their duals with respect to dualizing complexes.

\subsection{Dualizing complexes}
\label{ss:dualizing-complexes}
Let $D$ be a dualizing complex for $R$, normalized\footnote{In \cite{Hartshorne:1966, Roberts:1980}, a dualizing complex is normalized to be nonzero  in $[-d, 0]$.} so $D^i$ is nonzero only in the range $[0,d]$, where $d \colonequals \dim R$ and always with nonzero cohomology in degree 0.   Thus $D$ is an $R$-complex with $\hh D$ finitely generated, and $\rgam{\fm}D\simeq \Sigma^{-d}E$, where $E$ is the injective hull of $k$; see \cite[Chapter 2, \S3]{Roberts:1980}. For any $R$-complex $M$ set
\[
\ddual M \colonequals \Rhom RMD\,.
\]
One version of the local duality theorem is that the functor $M\mapsto \ddual M$ is a contravariant equivalence when restricted to $\mathrm{D^{b}}(\mathrm{mod}\, R)$, the bounded derived category of finitely generated $R$-modules; see \cite[Chapter 2, Theorem 3.5]{Roberts:1980}.
For $M$ in this subcategory, this gives  the last of following quasi-isomorphisms:
\begin{align*}
\Rhom R{\ddual M}E
	&= \Rhom R{\Rhom RMD}E\\
	&\simeq \Rhom R{\Rhom RMD}{\Sigma^{d} \rgam{\fm}D}\\
	&\simeq \Sigma^{d} \rgam{\fm}{ \Rhom R{\Rhom RMD}D}\\
	&\simeq \Sigma^{d} \rgam {\fm} M
\end{align*}
The rest are standard. Passing to cohomology yields the usual local duality:
\begin{equation}
\label{eq:duality}
\Hom R{\CH {i}{\ddual M}}E \cong \lch{d-i}{\fm}M \quad\text{for each $i$.}
\end{equation}
When $R$ is $\fm$-adically complete, one can apply Matlis duality to express $\CH i{\ddual M}$ as a dual of $\lch{d-i}{\fm}M$.

We also need to introduce a class of maps that will play an important role in the sequel: For any $R$-module $N$ let $\zeta^i_N$ denote the composition of maps
\begin{equation}
\label{eq:theone}
\Ext iRkN \xrightarrow{\ \cong\ } \Ext iRk{\rgam{\fm}N} \lra \Ext iRR{\rgam{\fm}N} \cong \lch i{\fm}N
\end{equation}
where the one in the middle is induced by the surjection $R\to k$. We will be particular interested in $\zeta^d_N$. If this map is nonzero, then $\dim_RN=\dim R$, but the converse does not hold.

\begin{proposition}
\label{pr:mcm-duality}
With  $D$ as above and  $M$ an $R$-complex with $\hh M$ finitely generated, set $N\colonequals \CH 0{\ddual M}$. Then $M$ is MCM if and only if $\ddual M \simeq N$ and the map $\zeta^d_N$  is nonzero, for $d=\dim R$.
\end{proposition}

\begin{proof}
 Given the hypothesis on the local cohomology on $M$, it follows that $\CH i{\ddual M}$ is nonzero for $i\ne 0$ and hence ${\ddual M}\simeq N$. Moreover, this quasi-isomorphism yields
\[
\rgam{\fm}N\simeq \rgam{\fm}{\Rhom RMD} \simeq \Rhom RM{\rgam{\fm}D} \simeq \Sigma^{-d} \Hom RME\,.
\]
Therefore the map \eqref{eq:theone} is induced by (to be precise, the degree $0$ component of the map in cohomology induced by) the map
\[
\Rhom Rk{\Hom RME} \lra \Hom RME
\]
By adjunction, the map above is
\[
\Rhom R{k\lotimes R M} E \lra \Hom RME
\]
That is to say, \eqref{eq:theone} is the Matlis dual of the map $\CH 0M\to \CH 0{k\lotimes RM}$. This justifies the claims.

Clearly, these steps are reversible: if $N$ is a finitely generated $R$-module such that the map \eqref{eq:theone} is nonzero, the $R$-complex $\Rhom RND$ is MCM.
\end{proof}

Here then is a way (and the only way) to construct MCM complexes when $R$ has a dualizing complex: Take a finitely generated $R$-module $N$ for which  $\zeta^d_N$ is nonzero; then the  complex $\Rhom RND$ is MCM.  It thus becomes important to understand the class of finitely generated $R$-modules for which the map $\zeta^d_N$ is nonzero.

To that end let $F$ be a minimal free resolution of $k$, and set
\[
\syzk\colonequals \Coker(F_{d+1}\to F_d)\,;
\]
this is the $d$th syzygy module of $k$. Since minimal free resolutions are isomorphic as complexes, this $\syzk$ is independent of the choice of resolution, up to an isomorphism. The canonical surjection $F\to F_{\ges d}$ gives a morphism in $\dcat R$:
\begin{equation}
\label{eq:ve}
\varepsilon \colon k\lra \shift^d \syzk\,.
\end{equation}
We view it as an element in $\Ext dRk{\syzk}$. The map $\zeta^d_{\syzk}$ below is from \eqref{eq:theone}.

\begin{lemma}
\label{le:upsilon}
One has  $\zeta^d_{\syzk}(\varepsilon)=0$ if and only if  $\zeta^d_{\syzk}=0$ if and only if $\zeta^d_N=0$ for all $R$-modules $N$.
\end{lemma}

\begin{proof}
Fix an $R$-module $N$. Any map $f$ in $\Hom R{\syzk}N$ induces a map
\[
f_*\colon \Ext dRk{\syzk} \lra \Ext dRkN\,.
\]
Let $F$ be a resolution of $k$ as above, defining $\syzk$. Any map $k\to \shift^d N$ in $\dcat R$ is represented by a morphism of complexes $F\to \shift^d N$, and hence factors through the surjection $F\to F_{\ges d}$, that is to say, the morphism $\varepsilon$. We deduce that any element of $\Ext dRkN$ is of the form $f_*(\varepsilon)$, for some $f$ in $\Hom R{\syzk}N$.

In particular,  $\Ext dRk{\syzk}$ is a generated by $\varepsilon$ as a left module over $\End_R(\syzk)$. This  observation, and the linearity of the $\zeta^d_{\syzk}$ with respect to $\End_R(\syzk)$, yields  $\zeta^d_{\syzk}=0$ if and only if $\zeta^d_{\syzk}(\varepsilon)=0$. Also each $f$ in $\Hom R{\syzk}N$ induces a commutative square
\[
\begin{tikzcd}
\Ext dRk{\syzk} \ar[d,"f_*" swap] \ar[r, "\zeta^d_{\syzk}"] & \lch d{\fm}{\syzk}  \ar[d,"\lch d{\fm}f"] \\
\Ext dRk{N} \ar[r, "\zeta^d_{N}"] & \lch d{\fm}{N}
\end{tikzcd}
\]
Thus if $\zeta^d_{\syzk}=0$ we deduce that $\zeta^d_N(f_*{\varepsilon})=0$. By varying $f$ we conclude from the discussion above that $\zeta^d_N=0$.
\end{proof}

We should record the following result immediately. It is one formulation  of the Canonical Element Theorem; see \cite[(3.15)]{Hochster:1983}. The ``canonical element" in question
is $\zeta^d_{\syzk}(\varepsilon)$; see Lemma~\ref{le:upsilon}.

\begin{theorem}
\label{th:cec}
For any noetherian local ring $R$, one has $\zeta^d_{\syzk}\ne 0$. \qed
\end{theorem}

Here then is first construction of an MCM $R$-complex.

\begin{corollary}
\label{co:MCM=CEC}
If $R$ has a dualizing complex the $R$-complex $\ddual{\syzk}$ is MCM.\qed
\end{corollary}

\begin{remark}
\label{re:MCM=CEC}
Suppose $R$ has a dualizing complex. Given Proposition~\ref{pr:mcm-duality} and Lemma~\ref{le:upsilon} it follows that $\ddual {\syzk}$ is MCM if and only there exists \emph{some} $R$-complex $M$ that is MCM.  Therefore, the Canonical Element Theorem, in all its various formulations~\cite{Hochster:1983}, is equivalent to the statement that  $R$ has an MCM R-complex!
\end{remark}

We now describe another way to construct an MCM complex. Let $D$ be a dualizing complex for $R$ and set  $\omega_R \colonequals \CH 0D$;
this is the \emph{canonical module} of $R$.

\begin{lemma}
\label{le:omega}
One has $\zeta^d_{\syzk}\ne 0$ if and only if $\zeta^d_{\omega_R}\ne 0$.
\end{lemma}

\begin{proof}
We write $\omega$ for $\omega_R$. Given Lemma~\ref{le:upsilon} we have to verify that if $\zeta^d_{\syzk}\ne 0$, then $\zeta^d_{\omega}\ne 0$. Let $E$ be an injective hull of $k$, the residue field of $R$. Since this is a faithful injective, there exists a map $\alpha\colon \lch d{\fm}{\syzk} \to E$ such that $\alpha\circ \zeta^d_{\syzk}\ne 0$.

It follows from local duality~\ref{eq:duality}, applied to $M=\syzk$, that $\alpha$ is induced by a morphism $f\colon \syzk\to D$; equivalently, an $R$-linear map $f\colon \syzk \to \omega$.  This gives the following commutative diagram
\[
\begin{tikzcd}
\Ext dRk{\syzk} \ar[d,"f_*" swap] \ar[r, "\zeta^d_{\syzk}"] & \lch d{\fm}{\syzk}  \ar[d,"\lch d{\fm}f" swap] \ar[dr,"\alpha"] \\
\Ext dRk{N} \ar[r, "\zeta^d_{\omega}" swap] & \lch d{\fm}{\omega} \ar[r] & E
\end{tikzcd}
\]
Since $\alpha\circ \zeta^d_{\syzk}\ne 0$ we conclude that $\zeta^d_{\omega}\ne 0$, as desired.
\end{proof}

\begin{corollary}
\label{co:omega}
If $R$ has a dualizing complex, the $R$-complex $\ddual{\omega_R}$ is MCM. \qed
\end{corollary}

The preceding result prompts a natural question.

\begin{question}
\label{qu:dualizing-CM}
When is the dualizing complex itself an MCM complex?

\medskip

Let $R$ be a local ring with a dualizing complex $D$, normalized as in \ref{ss:dualizing-complexes}.  The local cohomology of $D$ has the right properties, so, by
Proposition~\ref{pr:mcm-duality}, the $R$-complex $D$ is MCM precisely when  $\zeta^d_R$ is nonzero. Easy examples involving non-domains show that this is not always the case;
Dutta~\cite{Dutta:1994a} asked: \textit{Is $\zeta^d_R$ nonzero whenever $R$ is a complete normal domain?} Recently, Ma, Singh, and Walther~\cite{Ma/Singh/Walther:2020} constructed counterexamples.

On the other hand, when $R$ is \emph{quasi-Gorenstein}, that is to say, when $\omega_R$ is free, it follows from Corollary~\ref{co:omega} that $D$ is MCM.
\end{question}

Here is a broader question, also of interest, concerning the maps $\zeta^i_N$: It is easy to check that this is nonzero when $i=\depth_RN$. What conditions on $N$ ensure that this is the only $i$ for which it is true? By taking direct sums of modules of differing depths, we obtain modules $N$ with $\zeta^i_N$ nonzero for more than a single $i$.

\begin{example}
When $(R,\fm,k)$ is a regular local ring  and $N$ is a finitely generated $R$-module, then $N$ is Buchsbaum if and only if $\zeta^i_N$ is surjective for each $i < \dim_RN$. So any non-CM Buchsbaum $R$-module would  give an example.
\end{example}

\begin{remark}
\label{re:zeta-test}
Let $F\simeq k$ be a free resolution of $k$ and $\ul r\colonequals r_1,\dots,r_n$ elements such that $(\ul r)$ is primary to the maximal ideal. The canonical surjection $R/(\ul r)\to k$ lifts to a morphism of complexes $\kos {\ul r}R\to F$. Applying $\Hom R-N$ induces maps
\[
\Ext iRkN \lra \CH i{\ul r;N}
\]
It is easy to verify that $\zeta^i_N$ factors through this map. What is more, if $\ul s$ is another sequence of elements such that $\ul r\in (\ul s)$, then the map above factors as
\[
\Ext iRkN \lra \CH i{\ul s;N} \lra \CH i{\ul r;N}
\]
Thus if any of maps above are zero, so is $\zeta^i_N$.
\end{remark}

We would like to record a few more observations about MCM complexes.

\begin{remark}
\label{re:MCM=CEC2}
Let $(R,\fm,k)$ be an arbitrary noetherian local ring. Then its $\fm$-completion, $\widehat{R}$, has a dualizing complex, and hence an MCM $\widehat{R}$-complex, as discussed above. Since any MCM $\widehat{R}$-complex is a big Cohen-Macaulay complex over $R$, we conclude that $R$ has a big Cohen-Macaulay complex, and, in particular, a complex of maximal depth.
\end{remark}

\begin{remark}
\label{re:localization}
Assume $R$ has a dualizing complex and that $M$ is an MCM $R$-complex. It is easy to check using Proposition~\ref{pr:mcm-duality} that $M_{\fp}$ is an MCM $R_\fp$-complex for $\fp$ in $\spec R$, as long as condition (2) defining MCM complexes holds at $\fp$. For example, if $A$ is dg  $R$-algebra that is MCM as an $R$-complex, then since $A_\fp$ is a dg $R_\fp$-algebra, Lemma~\ref{le:mcm-dga} implies that it is an MCM $R_\fp$-complex.
\end{remark}

\begin{remark}
\label{re:mcm-module}
While MCM complexes have their uses, as the discussion in Section~\ref{se:max-depth} makes clear, they are not always a good substitute for MCM \emph{modules}. Indeed, in \cite[\S3]{Hochster:1975} Hochster proves if every local ring has an MCM module, then the Serre positivity conjecture on multiplicities is a consequence of the vanishing conjecture; see also \cite[\S4]{Hochster:2017}. Hochster's arguments cannot be carried out with MCM complexes in place of modules. The basic problem is this: Given a finite free complex $F$, over a local ring $R$, with homology of finite length, if $M$ is an MCM $R$-module, then $\hh{F\otimes_RM}$ is concentrated in at most one degree; this need not be the case when $M$ is an MCM complex. Indeed this is clear from Iversen's Amplitude inequality~\cite{Iversen:1977}, which is a reformulation of the New Intersection Theorem, and reads:
\[
\mathrm{amp}(F\lotimes RX) \ge \mathrm{amp}(X)
\]
where $F$ is any finite free complex with $\hh F\ne 0$ and $X$ is an $R$-complex with $\hh X$ bounded. Here $\mathrm{amp}(X)\colonequals \sup\HH{*}X - \inf \HH{*}X$, the \emph{amplitude} of $X$. By the way, the Amplitude Inequality holds even when $\hh X$ is unbounded~\cite{Foxby/Iyengar:2001}.
\end{remark}

\subsection{Via resolution of singularities}
\label{ss:resolve-singularities}
The constructions of MCM complexes described above are independent of the characteristic of the ring, but proving that they are MCM is a non-trivial task, for it depends on knowing
that one has MCM complexes to begin with; see Remark~\ref{re:MCM=CEC}. Next, we describe a complex that arises from a completely different source that one can prove is MCM independently. The
drawback is that it is restricted to algebras essentially of finite type and containing the rationals.  We first record a well-known observation about proper maps.

\begin{lemma}
\label{le:cdga}
Let $R$ be any commutative noetherian ring and $\pi\colon X\to \spec(R)$ a proper map from a noetherian scheme $X$. Viewed as an object in $\dcat R$ the complex $\rder \pi\mcO_X$ is equivalent to a dg algebra with cohomology graded-commutative and finitely generated. When $R$ contains a field of characteristic zero, the dg algebra itself can be chosen to be graded-commutative.
\end{lemma}

\begin{proof}
By \cite[Theorem~3.2.1]{Grothendieck:ega3a},  since $\mcO_X$ is coherent and $\pi$ is proper, $\rder\pi \mcO_X$ is coherent and hence its cohomology is finitely generated. Next, we explain why this complex is equivalent, in $\dcat R$, to a dg algebra. The idea is that $\mcO_X$ is a ring object in $\dcat{X}$ and there is a natural morphism
\[
\rder\pi\mathcal{F}\lotimes R\rder\pi{\mathcal{G}}\lra \rder\pi{(\mathcal{F}\lotimes X{\mathcal G})}
\]
so $\rder\pi \mcO_X$ is ring object in $\dcat R$.  One can realize this concretely as follows.

Let $\{U_i\}_{i=1}^n$ be an affine cover of $X$. Then the \v{C}ech complex computing $\rder\pi \mcO_X$ is equivalent to the total complex associated to the co-simplicial commutative ring
\[
\begin{tikzcd}
0\ar[r] & \prod_i \Gamma(X, U_i) \ar[r,shift left=1ex]\ar[r,shift right=1ex]
	& \prod_{i,j}\Gamma(X, U_i\cap U_j) \ar[r,shift left=1ex] \ar[r] \ar[r,shift right=1ex]
	& \prod_{i,j,k}\Gamma(X, U_i\cap U_j\cap U_k)
\end{tikzcd}
\]
It remains to point out that the Alexander-Whitney map makes the normalization of a co-simplicial ring a dg algebra, with graded-commutative cohomology. Moreover, over a field of characteristic zero, it is even quasi-isomorphic to a graded-commutative dg algebra.
\end{proof}

The statement of the next result, which is due to Roberts~\cite{Roberts:1980},  invokes the resolution of singularities in characteristic zero, established by Hironaka. The proof uses Grothendieck duality for projective maps~\cite{Hartshorne:1966} and the theorem of Grauert and Riemenschneider ~\cite{Grauert/Riemenschneider:1970} on the vanishing of cohomology. Given these, the calculation that is needed is standard; see the proof of \cite[Proposition~2.2 ]{Hartshorne/Ogus:1974} due to Hartshorne and Ogus. It will clear from the argument that it can be carried out also in any context where one has sufficient vanishing of cohomology; see \cite[Theorem~3.3]{Roberts:1980}.

\begin{proposition}
\label{pr:resolve-singularities}
Let $(R,\fm,k)$ be a noetherian local ring essentially of finite over a field of characteristic zero, and  $\pi\colon X \to \spec(R)$ a resolution of singularities. The $R$-complex $\rder\pi \mcO_X$ is MCM and equivalent to a graded-commutative dg algebra.
\end{proposition}

\begin{proof}
Given Lemmas~\ref{le:cdga} and ~\ref{le:mcm-dga} it  remains to verify that $\lch j{\fm}{\rder\pi \mcO_X} = 0$ for  $j \ne d$, where $d\colonequals\dim R$.  Let $\kappa \colon \spec(R)\to\spec(k)$ be the structure map and  set $D\colonequals \kappa^!(\mcO_k)$; this is then a dualizing complex for $R$, and $\pi^!D=\pi^!\kappa^!(\mcO_k)\cong \omega_X$, the dualizing sheaf for $X$. Since the $R$-complex $\rder\pi \mcO_X$ has finitely generated cohomology,  local duality~\ref{eq:duality} yields the first isomorphism below
\begin{align*}
\lch j{\fm}{\rder\pi O_X}
	&\cong {\Ext{d-j}R{\rder\pi \mcO_X}{D}}^\vee \\
	& \cong \Ext{d-j}X{\mcO_X}{\pi^! D}^\vee \\
	& =  {\CH{d-j}{X, \omega_X}}^\vee
\end{align*}
The second isomorphism is by coherent Grothendieck duality~\cite{Hartshorne:1966}. It remains to invoke the Grauert-Riemenschneider vanishing theorem~\cite{Grauert/Riemenschneider:1970} to deduce that the last module in the display is $0$ for all $j\ne d$.
\end{proof}

Here is a natural question, growing out of Proposition~\ref{pr:resolve-singularities}. A positive answer might have a bearing on the theory of multiplier ideals; see Theorem~\ref{th:multiplier}.

\begin{question}
\label{qu:cdga}
When $R$ contains a field of positive characteristic, or is of mixed characteristic, does it have an MCM $R$-complex that is also a dg algebra? What about a graded-commutative dg algebra?
\end{question}

\section{Applications to birational geometry}
\label{se:birational}

In this section we prove two celebrated results in birational geometry using MCM complexes constructed via Proposition~\ref{pr:resolve-singularities}. The first one generalizes Boutot's theorem on rational singularities \cite{Boutot:1987}; the argument is only a slight reworking of Boutot's proof, emphasizing the role of the derived push-forward as an MCM complex. Related circles of ideas can be found in the work of Bhatt, Koll\'ar, Kov\'acs, and Ma~\cite{Bhatt:2012, KollarShaferavich, Kovacs:2000, Ma:2018}.

\begin{theorem}
\label{th:boutot}
Let $\rho\colon Z\to \spec R$ be a map of schemes essentially of finite type over a field of characteristic $0$ such that $R\to \rder\rho\mcO_Z$ splits in $\dcat R$. If $Z$ has rational singularities, then so does $R$.
\end{theorem}

\begin{proof}
We may assume $(R,\fm)$ is local. Note that the condition implies $R\to\pi_*O_X$ is injective so in particular $R$ is reduced (as $X$ is reduced). Thus we can take $\pi\colon X\to\spec R$ to be a resolution of singularities. Then there is a (reduced) subscheme of $X\times_{\spec R}Z$ that is birational over $Z$ for each irreducible component of $Z$. Let $Y$ be a resolution of singularities of that subscheme. Thus there is a commutative diagram:
\[
\begin{tikzcd}
Y \ar[r] \ar[d, "\sigma"] & X \ar[d, "\pi"] \\
Z \ar[r, "\rho"] & \spec R.
\end{tikzcd}
\]
This induces a commutative diagram
\[
\begin{tikzcd}
R \ar[r] \ar[d] & \rder\rho\mcO_Z \ar[d,"\cong"] \\
\rder\pi\mcO_X \ar[r] & \rder\rho\rder\sigma\mcO_Y.
\end{tikzcd}
\]
The right vertical map is an isomorphism since $Z$ has rational singularities. Now since $R\to \rder\rho\mcO_Z$ splits in $D(R)$, chasing the diagram shows that $R\to \rder\pi\mcO_X$ splits in $\dcat R$. In particular, we know that the induced map
\[
\lch i{\fm}R \hookrightarrow \lch i{\fm}{\rder \pi\mcO_X}
 \]
is split-injective for all $i$. Because $\rder\pi\mcO_X$ is a MCM complex, by Proposition~\ref{pr:resolve-singularities}, it follows that $\lch i{\fm}{R}=0$ for $i<d$, that is to say, $R$ is Cohen-Macaulay. Finally, the Matlis dual of the injection above yields a surjective map $\pi_*\omega_X\twoheadrightarrow \omega_R$. Therefore $\pi_*\omega_X\cong \omega_R$ since $X\to \spec R$ is birational.

Putting these together yields $\omega_R^\bullet\cong \rder \pi\omega_X^\bullet$, where $\omega_R^\bullet$ and $\omega_X^\bullet$ are the normalized dualizing complex of $R$ and $X$ respectively. Applying $\Rhom{R}{-}{\omega_R^\bullet}$ and using Grothendieck duality yields $R\cong \rder \pi\mcO_X$. Thus $R$ has rational singularities.
\end{proof}

Here is an application.

\begin{corollary}
Suppose $(R,\Delta)$ is KLT, then $R$ has rational singularities.
\end{corollary}

\begin{proof}
Let $\pi\colon Y\to X=\spec R$ be a log resolution of $(R,\Delta)$. Since $(R, \Delta)$ is KLT, we know that $\lceil K_Y-\pi^*(K_X+\Delta)\rceil$ is effective and exceptional, thus
\[
R=\pi_*\mcO_Y(\lceil K_Y-\pi^*(K_X+\Delta)\rceil)=\rder\pi \mcO_Y(\lceil K_Y-\pi^*(K_X+\Delta)\rceil)\,,
\]
where the second equality follows from relative Kawamata-Viehweg vanishing~\cite[Theorem~9.4.1]{Lazarsfeld:2004}. Now we have
\[
R\to \rder\pi\mcO_Y \to \rder\pi\mcO_Y(\lceil K_Y-\pi^*(K_X+\Delta)\rceil)\cong R\,,
\]
that is to say, the map $R\to\rder\pi\mcO_Y$ splits in $\dcat R$. Theorem \ref{th:boutot} then implies $R$ has rational singularities.
\end{proof}

Our second application is a new proof of the subadditivity property of multiplier ideals \cite{Lazarsfeld:2004}. Our idea of using the MCM property of $\rder\pi\mcO_X$ to prove this comes from the analogous methods in positive and mixed characteristic \cite{Ma/Schwede:2018, Takagi:2006}.

\begin{theorem}
\label{th:multiplier}
Let $(A,\frak{m})$ be a noetherian regular local ring essentially of finite type over a field of characteristic $0$. Given ideals $\frak{a}, \frak{b}$ in $A$ and  numbers $s, t \in \mathbb{Q}_{\geq 0}$, one has $J(A, \frak{a}^s\frak{b}^t)\subseteq J(A, \frak{a}^s)J(A, \frak{b}^t)$.
\end{theorem}

\begin{proof}
We first claim that we may assume that $\frak{a}, \frak{b}$ are both principal ideals. For this we need an analog, for \emph{mixed} multiplier ideals, \cite[Proposition 92.26]{Lazarsfeld:2004}.  The argument is the same and we now sketch it.  Indeed, fix general elements $f_1, \dots, f_k$ in $\frak{a}$ and $g_1, \dots, g_l$ in $\frak{b}$ for $k > s$ and $l > t$ and set
\begin{align*}
D_1 &= {1 \over k}\sum \mathrm{Div}(f_i) = {1 \over k} \mathrm{Div}(\prod \mathrm{Div}(f_i)) \\
 D_2 &= { 1 \over l} \sum \mathrm{Div}(g_i) = {1 \over l} \mathrm{Div}(\prod \mathrm{Div}(g_i))\,.
 \end{align*}
For a log resolution $\pi\colon  X \to \mathrm{Spec} A$ of $(A, \frak{a}, \frak{b})$ with $\mcO_X(-F) = \frak{a} \cdot \mcO_X$ and $\mcO_X(-G) = \frak{b} \cdot \mcO_X$, we have that
\[
    \pi^{-1}_* \mathrm{Div}(f_i) + F_{\mathrm{exc}} = \pi^* \mathrm{Div}(f_i) \;\;\text{ and }\;\; \pi^{-1}_* \mathrm{Div}(g_i) + G_{\mathrm{exc}} = \pi^* \mathrm{Div}(g_i).
\]
where $\pi^{-1}_*$ denotes the strict transform and $F_{\mathrm{exc}}$ and $G_{\mathrm{exc}}$ are the $\pi$-exceptional parts of $F$ and $G$.  Since the $f_i$ and $g_i$ are generic, the associated divisors and their strict transforms are reduced.  A straightforward computation then shows that
\[
    \lfloor sF \rfloor = \left\lfloor {s \over k} \sum \pi^* \mathrm{Div}(f_i) \right\rfloor \;\;\text{ and }\;\;
    	\lfloor tG \rfloor = \left\lfloor {t \over l} \sum \pi^* \mathrm{Div}(g_i) \right\rfloor.
\]
Thus $J(A, \frak{a}^s \frak{b}^t) = J(A, (\prod f_i)^{s/k} (\prod g_i)^{t/l})$, and likewise $J(A, \frak{a}^s) = J(A, (\prod f_i)^{s/k})$ and $J(A, \frak{b}^t) = J(A, (\prod g_i)^{t/l})$.  Therefore we may assume that $\frak{a}$ and $\frak{b}$ are principal.

Now we assume $\frak{a}=(f)$ and $\frak{b}=(g)$. Let $R$ be the normalization of $A[f^{1/d_s}, g^{1/d_t}]$ where $d_s$ and $d_t$ are the denominators of $s$ and $t$; thus $f^s, g^t$ are elements in $R$. Let $\pi\colon X\to \spec R$ be a resolution of singularities. Thus $X\to \spec A$ is a regular alteration; we write $\pi$ also for this map.

In what follows, to simplify notation, we write $E$ for $\lch d{\fm}A$. Given an element $r\in R$ let $0^{r}_{E}$ for the kernel of the composite map
\[
E=\lch d{\fm}A \lra \lch d{\fm} R  \xra{\ r\ } \lch d{\fm}R \lra \lch d{\fm}{\rder\pi\mcO_X}
\]
Now suppose that a power $r^m$ of $r$ lives in $A$ (for instance $r = f^{s}$ or $r = g^{t}$).  Then by \cite[Theorem 8.1]{BlickleSchwedeTucker} we have that $\mathrm{Tr}(J(\omega_R, r)) = J(A, (r^m)^{1/m})$.  By local duality it is easy to see that
\[
J(A,(r^{m})^{1/m}) = \ann_A 0^{r}_{E}\,.
\]
In particular, $J(A, f^s)=\ann_A 0^{f^s}_{E}$ and $J(A, g^t)=\ann_A 0^{g^t}_{E}$.

We next claim that the following inclusion holds:
\begin{equation}
\label{eq:claim}
\{\eta\in E \mid J(A, f^s)\cdot \eta \subseteq 0^{g^t}_{E}\}\subseteq 0^{f^sg^t}_{E}.
\end{equation}
Indeed,  suppose $J(A, f^s)\eta\subseteq 0^{g^t}_{E}$, then $J(A, f^s)\cdot g^t\eta=0$ in $\lch d{\fm}{\rder\pi\mcO_X}$. Note that $g^t\eta$ makes sense in $\lch d{\fm}{\rder\pi\mcO_X}$ as the latter is a module over $R$. Thus
\[
g^t\eta\in \ann_{\lch d{\fm}{\rder\pi\mcO_X}}J(A, f^s)\cong \Hom{A}{A/J(A, f^s)}{\lch d{\fm}{\rder\pi\mcO_X}}\,.
\]
Next, because
\[
h^{-i}(\rder\pi\mcO_X\lotimes A E) \cong \lch {d-i}{\fm}{\rder\pi\mcO_X}=0
\]
for all $i\ge 1$, we know that $ g^t\eta$ is in the module
\begin{eqnarray*}
 \Hom{A}{\frac{A}{J(A, f^s)}}{\lch d{\fm}{\rder\pi\mcO_X}}
   &\cong & h^0\left(\Rhom{A}{\frac{A}{J(A, f^s)}}{\rder\pi\mcO_X\lotimes A E}\right)\\
   &\cong& h^0\left( \rder\pi\mcO_X \lotimes A \Rhom{A}{\frac{A}{J(A,f^s)}}{E}\right)\\
   &\cong& h^0\left( \rder\pi\mcO_X \lotimes A \ann_{E}J(A, f^s)\right)
\end{eqnarray*}
where the second isomorphism follows from \cite[Proposition 1.1 (4)]{Foxby:1977}, noting that $A$ is regular thus every bounded complex is isomorphic to a bounded complex of flat modules in $\dcat A$, and the third isomorphism follows from the fact that $E$ is an injective $A$-module.

Consider following composite map; again, the second multiplication by $f^s$ map makes sense since we can view $\rder\pi\mcO_X$ as a complex over $R$ and not merely over $A$:
\[
E\to h^0\left( \rder\pi\mcO_X \lotimes A {E}\right) \xrightarrow{\cdot f^s} h^0\left( \rder\pi\mcO_X \lotimes A {E}\right)
\]
Its kernel is $\ann_{E}J(A, f^s)$, by Matlis duality. Thus the composition of the natural induced maps
\[
 \rder\pi\mcO_X \lotimes A \ann_{E}J(A, f^s) \to  \rder\pi\mcO_X \lotimes A {E}
 	 \xrightarrow{\cdot f^s}  \rder\pi\mcO_X \lotimes A {E}
\]
is  zero in $h^0$. In particular, since $g^t\eta$ is in $h^0$ of the source of this composite map, we deduce that, viewed as an element in target, namely in
\[
h^0\left( \rder\pi\mcO_X \lotimes A {E}\right)\cong \lch d{\fm}{\rder\pi\mcO_X}
\]
it is killed by $f^s$. Therefore $f^sg^t\eta=0$ in $\lch d{\fm}{\rder\pi\mcO_X}$ and hence $\eta\in 0^{f^sg^t}_{E}$.

This justifies \eqref{eq:claim}.

Finally, for any $z\in \ann_{E}J(A, f^s)J(A, g^t)$, we have $J(A, f^s)z\subseteq 0^{g^t}_{E}$ and thus $z\in 0^{f^sg^t}_{E}$ by \eqref{eq:claim}. Therefore $$\ann_{E}J(A, f^s)J(A, g^t)\subseteq 0^{f^sg^t}_{E}$$ and hence by Matlis duality $J(A, f^sg^t)\subseteq J(A, f^s)J(A, g^t)$.
\end{proof}


\end{document}